\documentclass{amsart}

\usepackage[margin=1.5in]{geometry} 
\usepackage[english]{babel} 
\usepackage[utf8]{inputenc} 
\usepackage[T1]{fontenc} 
\usepackage{lmodern} 
\usepackage{exscale} 
\usepackage{microtype} 

\usepackage{latexsym,amsxtra,amscd,amsmath,amssymb,amsthm,stmaryrd,mathtools,leftindex} 
\usepackage{bbold,bbm} 

\usepackage[dvipsnames]{xcolor} 
\usepackage{graphics} 
\usepackage{longtable} 
\usepackage{colortbl, booktabs} 
\usepackage{multicol, multirow} 

\usepackage{hyperref} 
\hypersetup{
    colorlinks=true,
    linkcolor=blue,
    filecolor=blue,      
    urlcolor=cyan,
    linktocpage=true
}

\usepackage{enumitem} 
\usepackage{comment} 
\usepackage[normalem]{ulem} 

\usepackage{tikzsymbols, pifont, bbding, skull} 
\usepackage[clock]{ifsym} 

\usepackage{amsfonts} 

\usepackage[all]{xy} 
\usepackage{tikz,tikz-cd} 

\makeatletter
\@namedef{subjclassname@2020}{%
  \textup{2020} Mathematics Subject Classification}
 
\makeatother

\newcommand{\stkout}[1]{\ifmmode\text{\sout{\ensuremath{#1}}}\else\sout{#1}\fi}

\newtheorem{theorem}{Theorem}[section]
\newtheorem{lemma}[theorem]{Lemma}
\newtheorem{proposition}[theorem]{Proposition}
\newtheorem{corollary}[theorem]{Corollary}

\theoremstyle{definition}
\newtheorem{definition}[theorem]{Definition}
\newtheorem{remark}[theorem]{Remark}

\newtheorem{example}[theorem]{Example}

\theoremstyle{remark}

\numberwithin{equation}{section}

\begin{document}

\title{Nilpotent graphs over skew PBW extensions}
\author{Sebasti\'an Higuera}
\address{Higuera: Universidad ECCI, Bogot\'a, D. C., Colombia}
\email{shiguerar@ecci.edu.co}

\author{Armando Reyes}
\address{Reyes: Universidad Nacional de Colombia, Bogot\'a, D.C., Colombia}
\email{mareyesv@unal.edu.co}

\dedicatory{Dedicated to Edgar Higuera}

\begin{abstract} 
     We investigate the diameter and girth of the nilpotent graph for skew PBW extensions over $2$-primal rings, generalizing similar results on skew polynomial rings. Under certain compatibility conditions, we establish bounds for the diameter of the nilpotent graph and prove invariance of the girth under polynomial extensions. 
\end{abstract}

\subjclass[2020]{05C12, 05C20, 13A99, 16S15, 16S30, 16S32, 16S36, 16S38 16U99}
\keywords{Nilpotent graph, diameter, compatible ring, skew polynomial ring, skew PBW extension, 2-primal ring.}

\maketitle	



\section{Introduction}\label{ch0}
\medskip

Throughout the paper, every ring is associative (not necessarily commutative) with identity unless otherwise stated. We write $Z_l(R),\ Z_r(R)$, $Z(R)^*$ and $U(R)$ to denote the set of all left zero divisors of $R$, the set of all right zero divisors of $R$, the set of non-zero zero divisors of $R$ and the set of all units of $R$, respectively. $P(R)$, $\mathrm{nil}^{*}(R)$ and $\mathrm{nil}(R)$ denote the prime radical of $R$, the upper nilradical of $R$ (that is, the sum of all nil ideals of $R$) and the set of nilpotent elements of $R$, respectively. If $P(R) = \mathrm{nil}(R)$ then $R$ is called \emph{2-primal}, and if $\mathrm{nil}^{*}(R) = \mathrm{nil}(R)$ then $R$ is an \emph{NI} ring. 

A graph $\Gamma$ is \emph{connected} if for every pair of distinct vertices $u$ and $v$ there exists a finite sequence of distinct vertices $v_1 = u, v_2, \dotsc, v_m = v$ such that each pair $\{v_i, v_{i+1}\}$ is an edge; the sequence is called a \emph{path}. If $a$ and $b$ are two distinct vertices of graph $\Gamma$, then the \emph{distance} between $a$ and $b$, denoted by $d(a,b)$, is the length of a shortest path connecting $a$ and $b$, if such a path exists; otherwise $d(a, b) = \infty$. The \emph{diameter} of a graph $\Gamma$ is defined as
\[
{\rm diam}(\Gamma) = \sup\{ d(a,b) \mid a\ {\rm and}\ b\ {\rm are\ distinct\ vertices\ of}\ \Gamma\}.
\]
If $V=\{v_1, v_2, \dotsc, v_m\}$ is a set of distinct vertices such that each pair $\{v_i, v_{i+1}\}$ is an edge and $v_1 = v_m$, then $V$ is a \emph{cycle} of length $m-1$. The \emph{girth} of a graph $\Gamma$, denoted by ${gr}(\Gamma)$, is the length of the shortest cycle in $\Gamma$, provided $\Gamma$ contains a cycle; otherwise ${gr}(\Gamma) = \infty$.

Beck \cite{Beck1988} introduced the notion of \emph{zero-divisor graph} on commutative rings. In his work, he considered all elements of $R$ as vertices of the zero-divisor graph. Anderson and Livingston \cite{AndersonLivingston1999} studied the (undirected) zero-divisor graph $\Gamma(R)$ of $R$, whose vertices are the non-zero zero-divisors of $R$, and considering distinct vertices $a$ and $b$ as adjacent if and only if $ab = 0$. They proved that $\Gamma(R)$ is a finite graph with at least one vertex if and only if $R$ is finite and it is not a field \cite[Theorem 2.2]{AndersonLivingston1999} and $\Gamma(R)$ is connected with $\rm{diam}(\Gamma(R)) \leq 3$ \cite[Theorem 2.3]{AndersonLivingston1999}. They also showed that ${gr}(\Gamma(R)) \leq 4$ if $\Gamma(R)$ contains a cycle \cite[Theorem 2.4]{AndersonLivingston1999}. Some properties of zero-divisor graphs for commutative rings have been studied by different authors \cite{AndersonNaseer1993, Axteletal2005}.

In the noncommutative setting, Redmond \cite{Redmond2002} defined the \emph{zero-divisor graph} $\Gamma(R)$, where the vertex set is $Z^{*}(R) = Z(R) \setminus \{0\}$, and distinct vertices $a$ and $b$ are connected by an edge if and only if $ab = 0$ or $ba = 0$ \cite[Definition 3.1]{Redmond2002}. Some properties of these graphs have been studied by different authors \cite{AlhevazKiani2014, AndersonMulay2007, KuzminaMaltsev2015, Zhuravlevetal2013, Lucas2006}. In the setting of skew polynomial rings introduced by Ore \cite{Ore1933}, Hashemi {et al.} \cite{HashemiAbdiAlhevaz2017, HashemiAbdiAlhevaz2019, HashemiAbdiAlhevazSu2020, HashemiAmirjanAlhevaz2017} investigated the relationship between the zero-divisor graph and the algebraic properties of these noncommutative rings by considering some conditions on rings such as reversibility and compatibility. 

If $R$ has no non-zero nilpotent elements, then $R$ is called \emph{reduced}. If $ab = 0$ implies that $ba = 0$, for all $a, b\in R$ then $R$ is {\em reversible}. Finally, $R$ is {\em symmetric} provided $abc = 0$ implies that $acb = 0$, for all $a, b, c\in R$. Commutative rings are symmetric, and symmetric rings are reversible, but the converses do not hold (\cite[Examples I.5 and II.5]{AndersonCamillo1999}, and \cite[Examples 5 and 7]{Marks2002}). In addition, every reduced ring is symmetric \cite[Lemma 1.1]{Shin1973}, but the converse is not necessarily true \cite[Example II.5]{AndersonCamillo1999}. If $\sigma$ is an endomorphism of $R$ and for each $r\in R$, we have that $r\sigma(r)=0$ implies that $r=0$, then $R$ is called {\em rigid}, and if there exists a rigid endomorphism $\sigma$ of $R$, then $R$ is $\sigma$-{\em rigid}. Annin in his Ph.D. Thesis \cite{Annin2002PhD} (see also \cite{Annin2004, HashemiMoussavi2005}) or Hashemi and Moussavi \cite{HashemiMoussavi2005} introduced compatible rings as a generalization of $\sigma$-rigid rings and reduced rings. $R$ is {\it $\sigma$-compatible} if for each $a, b \in R$, then $ab = 0$ if and only if $a\sigma(b) = 0$; $R$ is {\it $\delta$-compatible} if for each $a, b \in R$, we obtain that $ab = 0$ implies that $a\delta(b) = 0$, and if $R$ is both $\sigma$-compatible and $\delta$-compatible, then $R$ is called {\it $(\sigma,\delta)$-compatible}. 


Hashemi {et al.} \cite{HashemiAmirjanAlhevaz2017} investigated the connection between the ring-theoretic properties of a skew polynomial ring $R[x;\sigma,\delta]$ and the graph-theoretic properties of its zero-divisor graph $\Gamma(R[x;\sigma,\delta])$. For example, they characterized the possible diameters of $\Gamma(R[x;\sigma,\delta])$ in terms of $\Gamma(R)$ when $R$ is reversible and $(\sigma,\delta)$-compatible. They also described associative rings for which all skew polynomial zero-divisor graphs are complete. Furthermore, they proved that the diameter of the graph of a skew polynomial ring depends on the properties of $R$ (such as being reduced, having minimal primes, or the ideal $Z(R)$) \cite[Theorem 2.7]{HashemiAmirjanAlhevaz2017}. Additionally, some properties of $R$ can be determined by knowing the diameters of $R$ and $R[x;\sigma, \delta]$ \cite[Theorem 2.8]{HashemiAmirjanAlhevaz2017}. Finally, they also provided upper and lower bounds for the diameter of $R$ in relation to its skew polynomial ring assuming that $R$ is reversible and $(\sigma,\delta)$-compatible.

Chen \cite{Chen2003} defined a kind of graph structure of rings where the vertex set consists of all elements of $R$ and two distinct vertices $x$ and $y$ are called \emph{adjacent} if and only if $xy \in \text{nil}(R)$. Li and Li \cite{LiLi2010} defined the \emph{nilpotent graph} $\Gamma_N(R)$, with vertex set $Z_N^*(R)$ and where two distinct vertices $x$ and $y$ are adjacent if and only if $xy \in \text{nil}(R)$. It is clear that the zero-divisor graph $\Gamma(R)$ is a subgraph of $\Gamma_N(R)$. Nikmehr and Khojasteh \cite{Nikmehretal2013} investigated the diameter and girth of the nilpotent graph for matrix rings over noncommutative rings, while Nikmehr and Azadi \cite{NikmehrAzadi2020} studied these properties for skew polynomial rings and proved that if $R$ is $\sigma$-compatible and symmetric, then $2 \leq \text{diam}(\Gamma_N(R[x;\sigma])) \leq 3$ \cite[Theorem 2.10]{NikmehrAzadi2020}, and that if $\Gamma_N(R)$ contains a cycle, then ${gr}(\Gamma_N(R)) = {gr}(\Gamma_N(R[x;\sigma]))$ \cite[Theorem 2.13]{NikmehrAzadi2020}.

Thinking about the above results on skew polynomial rings and nilpotent graphs, a natural task is to investigate graph-theoretic properties preserved under ring-theoretic constructions over non-commutative rings of polynomial type. Our purpose in this paper is to study the nilpotent graph of a {\em skew Poincar\'e-Birkhoff-Witt extension} ({\em SPBW} for short). The SPBW extension were defined by Gallego and Lezama \cite{LezamaGallego2011} as a generalization of the PBW extensions considered by Bell and Goodearl \cite{BellGoodearl1988} and skew polynomial rings of injective type introduced by Ore \cite{Ore1933}. Several authors have shown that SPBW extensions also generalize families of noncommutative algebras, such as 3-dimensional skew polynomial algebras introduced by Bell and Smith \cite{BellSmith1990}, diffusion algebras considered by Isaev {et al.} \cite{IsaevPyatovRittenberg2001}, ambiskew polynomial rings introduced by Jordan (see \cite{Jordan1993} and references therein), almost normalizing extensions defined by McConnell and Robson \cite{McConnellRobson2001} and skew bi-quadratic algebras recently introduced by Bavula \cite{Bavula2023}. For more details on the SPBW extensions, see \cite{Fajardoetal2020, HigueraRamirezReyes2024, HigueraReyes2022, LezamaReyes2014}.

The paper is organized as follows. In Section \ref{Definitions}, we recall fundamental definitions and establish key properties of SPBW extensions and $(\Sigma, \Delta)$-compatible rings. Section \ref{Nilpotentgraph} presents our main results: Theorems \ref{PBWTheorem2.9}, \ref{PBWTheorem2.10}, \ref{PBWTheorem2.11}, and \ref{PBWTheorem2.13} characterize the diameter and girth of nilpotent graphs associated with these extensions, generalizing previous work on skew polynomial rings $R[x;\sigma]$ by Nikmehr and Azadi \cite{NikmehrAzadi2020}. Section \ref{Examplespaper} illustrates our findings with examples of noncommutative algebras that are not skew polynomial rings.

Throughout the paper, $\mathbb{N}$, $\mathbb{Z}$, $\mathbb{R}$, and $\mathbb{C}$ denote the standard numerical systems, with $\mathbb{N}$ including the zero element. The symbol $\Bbbk$ denotes a field, and $\Bbbk^* := \Bbbk \setminus \{0\}$.

\section{Definitions and elementary properties}\label{Definitions}

\begin{definition}[{\cite[Definition 1]{LezamaGallego2011}}] \label{gpbwextension}
A ring $A$ is called a {\em skew PBW} ({\em SPBW}) {\em extension} {\em of} $R$, which is denoted by $A:=\sigma(R)\langle
x_1,\dots,x_n\rangle$, if the following conditions hold:
\begin{enumerate}
\item[\rm (i)]$R$ is a subring of $A$ sharing the same identity element.

\item[\rm (ii)] There exist finitely many elements $x_1,\dots ,x_n\in A$ such that $A$ is a left free $R$-module, with basis the set of standard monomials
\begin{center}
${\rm Mon}(A):= \{x^{\alpha}:=x_1^{\alpha_1}\cdots
x_n^{\alpha_n}\mid \alpha=(\alpha_1,\dots ,\alpha_n)\in
\mathbb{N}^n\}$.
\end{center}

We consider $x^0_1\cdots x^0_n := 1 \in {\rm Mon}(A)$.
\item[\rm (iii)] For every $1\leq i\leq n$ and any non-zero element $r\in R$, there exists a non-zero element $c_{i,r}\in R$ such that $x_ir-c_{i,r}x_i\in R$.
\item[\rm (iv)] For $1\leq i,j\leq n$, there exists a non-zero element $d_{i,j}\in R$ such that
\[
x_jx_i-d_{i,j}x_ix_j\in R+Rx_1+\cdots +Rx_n,
\]

i.e. there exist elements $r_0^{(i,j)}, r_1^{(i,j)}, \dotsc, r_n^{(i,j)} \in R$ with
\begin{center}
$x_jx_i - d_{i,j}x_ix_j = r_0^{(i,j)} + \sum_{k=1}^{n} r_k^{(i,j)}x_k$.    
\end{center}
\end{enumerate}
\end{definition}

Since ${\rm Mon}(A)$ is a left $R$-basis of $A$, this implies that the elements $c_{i,r}$ and $d_{i, j}$ are unique, and every non-zero element $f \in A$ can be uniquely expressed as $f = a_0 + a_1X_1 + \cdots + a_mX_m$ with $a_i \in R$, $X_0=1$, and $X_i \in \text{Mon}(A)$, for $0 \leq i \leq m$ \cite[Remark 2]{LezamaGallego2011}. 

If $A=\sigma(R)\langle x_1,\dots,x_n\rangle$ is an SPBW extension of $R$, there exist an injective endomorphism $\sigma_i$ of $R$ and a $\sigma_i$-derivation $\delta_i$ of $R$ such that $x_ir=\sigma_i(r)x_i+\delta_i(r)$, for all $1\leq i\leq n$ and $r\in R$ \cite[Proposition 3]{LezamaGallego2011}. The notation $\Sigma:=\{\sigma_1,\dots,\sigma_n\}$ and $\Delta:=\{\delta_1,\dots,\delta_n\}$ denotes the corresponding sets of injective endomorphisms and $\sigma_i$-derivations, respectively.

\begin{definition}\label{quasicommutative}
If $A=\sigma(R)\langle x_1,\dots,x_n\rangle$ is an SPBW extension of $R$, then:
\begin{itemize}
    \item[{\rm (i)}] {\cite[Definition 4]{LezamaGallego2011}} $A$ is {\it quasi-commutative} if the conditions ${\rm (iii)}$ and ${\rm (iv)}$ formulated in Definition \ref{gpbwextension} are replaced by the following: 
\begin{enumerate}
    \item[(iii')] For every $1 \leq i \leq n$ and non-zero element $r \in R$ there exists a non-zero element $c_{i,r} \in R$ such that $x_ir = c_{i,r}x_i$.
    
\item[(iv')] For every $1 \leq i, j \leq n$ there exists a non-zero element $d_{i,j} \in R$ such that $x_jx_i = d_{i,j}x_ix_j$.    
\end{enumerate}

    \item[{\rm (ii)}] \cite[Definition 4]{LezamaGallego2011} $A$ is {\it bijective} if  $\sigma_i$ is bijective for all $1 \leq i \leq n$, and $d_{i,j}$ is invertible for every $1 \leq i <j \leq n$.
    \item [\rm (iii)] \cite[Definition 2.3]{AcostaLezamaReyes2015} If $\sigma_i={\rm id}_R$ is the identity map of $R$ for all $1\le i \le n$, then $A$ is called of \textit{derivation type}. If $\delta_i=0$ for all $1\le i \le n$, then $A$ is of \textit{endomorphism type}.
\end{itemize}
\end{definition}

\begin{remark}[{\cite[Section 3]{LezamaGallego2011}}]\label{definitioncoefficients}
If $A=\sigma(R)\langle x_1,\dots,x_n\rangle$ is an SPBW extension of $R$, then: 
\begin{enumerate}
\item[\rm (1)] For any non-zero element $\alpha=(\alpha_1,\dots,\alpha_n)\in \mathbb{N}^n$, we will write 
$\sigma^{\alpha}:=\sigma_1^{\alpha_1}\circ \dotsb \circ \sigma_n^{\alpha_n}$, $\delta^{\alpha} = \delta_1^{\alpha_1} \circ \dotsb \circ \delta_n^{\alpha_n}$, where $\circ$ represents the usual function composition and 
$|\alpha|:=\alpha_1+\cdots+\alpha_n$. If
$\beta=(\beta_1,\dots,\beta_n)\in \mathbb{N}^n$, then
$\alpha+\beta:=(\alpha_1+\beta_1,\dots,\alpha_n+\beta_n)$.

\item[\rm (2)] Let $\succeq$ be a total order defined on ${\rm Mon}(A)$. If $x^{\alpha}\succeq x^{\beta}$ but $x^{\alpha}\neq x^{\beta}$, we write $x^{\alpha}\succ x^{\beta}$. If $f$ is a non-zero element of $A$, then we use expressions as $f=\sum_{i=0}^ma_iX_i$, with $a_i\in R$, and $X_m\succ \dotsb \succ X_1$. With this notation, we define ${\rm
lm}(f):=X_m$, the \textit{leading monomial} of $f$; ${\rm
lc}(f):=a_m$, the \textit{leading coefficient} of $f$; ${\rm
lt}(f):=a_mX_m$, the \textit{leading term} of $f$; ${\rm exp}(f):={\rm exp}(X_m)$, the \textit{order} of $f$. Notice that $\deg(f):={\rm max}\{\deg(X_i)\}_{i=1}^m$. Finally, if $f=0$, then
${\rm lm}(0):=0$, ${\rm lc}(0):=0$, ${\rm lt}(0):=0$. We also
consider $X\succ 0$ for any $X\in {\rm Mon}(A)$. Thus, we extend $\succeq$ to ${\rm Mon}(A)\cup \{0\}$.
\end{enumerate}
\end{remark}

\begin{proposition}[{\cite[Theorem 7]{LezamaGallego2011}}]\label{coefficientes}
If $A$ is a polynomial ring with respect to the set of indeterminates $\{x_1,\dots,x_n\}$ and $R$ is the coefficient ring, then $A$ is an SPBW  extension of $R$ if and only if the following conditions hold:
\begin{enumerate}
\item[\rm (1)] For each $x^{\alpha}\in {\rm Mon}(A)$ and every non-zero element $r\in R$, there exist unique elements $r_{\alpha}:=\sigma^{\alpha}(r)\in R\ \backslash\ \{0\}$ and $p_{\alpha ,r}\in A$ such that $x^{\alpha}r=r_{\alpha}x^{\alpha}+p_{\alpha, r}$,  where $p_{\alpha ,r}=0$, or $\deg(p_{\alpha ,r})<|\alpha|$ if
$p_{\alpha , r}\neq 0$. If $r$ is left invertible,  so is
$r_\alpha$.

\item[\rm (2)]For each $x^{\alpha},x^{\beta}\in {\rm Mon}(A)$,  there exist unique elements $c_{\alpha,\beta}\in R\ \backslash\ \{0\}$ and $p_{\alpha,\beta}\in A$ such that $x^{\alpha}x^{\beta}=c_{\alpha,\beta}x^{\alpha+\beta}+p_{\alpha,\beta}$, where $c_{\alpha,\beta}$ is left invertible, $p_{\alpha,\beta}=0$, or $\deg(p_{\alpha,\beta})<|\alpha+\beta|$ if
$p_{\alpha,\beta}\neq 0$.
\end{enumerate}
\end{proposition}

\begin{proposition}[{\cite[Proposition 2.9 and Remark 2.10 iv)]{Reyes2015Baer}}] \label{Reyes2015Proposition2.9} If $\sigma(R)\langle x_1,\dots,x_n\rangle$ is an SPBW extension of $R$, then
\begin{itemize}
    \item[{\rm (1)}] If $\alpha=(\alpha_1,\dotsc, \alpha_n)\in \mathbb{N}^{n}$ and $r\in R$ with $r\neq 0$, then  
{\footnotesize{\begin{align*}
x^{\alpha}r = &\ x_1^{\alpha_1}x_2^{\alpha_2}\dotsb x_{n-1}^{\alpha_{n-1}}x_n^{\alpha_n}r = x_1^{\alpha_1}\dotsb x_{n-1}^{\alpha_{n-1}}\biggl(\sum_{j=1}^{\alpha_n}x_n^{\alpha_{n}-j}\delta_n(\sigma_n^{j-1}(r))x_n^{j-1}\biggr)\\
&\ + x_1^{\alpha_1}\dotsb x_{n-2}^{\alpha_{n-2}}\biggl(\sum_{j=1}^{\alpha_{n-1}}x_{n-1}^{\alpha_{n-1}-j}\delta_{n-1}(\sigma_{n-1}^{j-1}(\sigma_n^{\alpha_n}(r)))x_{n-1}^{j-1}\biggr)x_n^{\alpha_n}\\
&\ + x_1^{\alpha_1}\dotsb x_{n-3}^{\alpha_{n-3}}\biggl(\sum_{j=1}^{\alpha_{n-2}} x_{n-2}^{\alpha_{n-2}-j}\delta_{n-2}(\sigma_{n-2}^{j-1}(\sigma_{n-1}^{\alpha_{n-1}}(\sigma_n^{\alpha_n}(r))))x_{n-2}^{j-1}\biggr)x_{n-1}^{\alpha_{n-1}}x_n^{\alpha_n}\\
&\ + \dotsb + x_1^{\alpha_1}\biggl( \sum_{j=1}^{\alpha_2}x_2^{\alpha_2-j}\delta_2(\sigma_2^{j-1}(\sigma_3^{\alpha_3}(\sigma_4^{\alpha_4}(\dotsb (\sigma_n^{\alpha_n}(r))))))x_2^{j-1}\biggr)x_3^{\alpha_3}x_4^{\alpha_4}\dotsb x_{n-1}^{\alpha_{n-1}}x_n^{\alpha_n} \\
&\ + \sigma_1^{\alpha_1}(\sigma_2^{\alpha_2}(\dotsb (\sigma_n^{\alpha_n}(r))))x_1^{\alpha_1}\dotsb x_n^{\alpha_n}, \ \ \ \ \ \ \ \ \ \ \sigma_j^{0}:={\rm id}_R\ \ {\rm for}\ \ 1\le j\le n.
\end{align*}}}
    
\item[{\rm (2)}] If $a_i, b_j\in R$ and $X_i:=x_1^{\alpha_{i1}}\dotsb x_n^{\alpha_{in}},\ Y_j:=x_1^{\beta_{j1}}\dotsb x_n^{\beta_{jn}}$, when we compute every summand of $a_iX_ib_jY_j$ we obtain products of the coefficient $a_i$ with several evaluations of $b_j$ in $\sigma$'s and $\delta$'s depending on the coordinates of $\alpha_i$. This assertion follows from the expression:
\begin{align*}
a_iX_ib_jY_j = &\ a_i\sigma^{\alpha_{i}}(b_j)x^{\alpha_i}x^{\beta_j} + a_ip_{\alpha_{i1}, \sigma_{i2}^{\alpha_{i2}}(\dotsb (\sigma_{in}^{\alpha_{in}}(b_j)))} x_2^{\alpha_{i2}}\dotsb x_n^{\alpha_{in}}x^{\beta_j} \\
&\ + a_i x_1^{\alpha_{i1}}p_{\alpha_{i2}, \sigma_3^{\alpha_{i3}}(\dotsb (\sigma_{{in}}^{\alpha_{in}}(b_j)))} x_3^{\alpha_{i3}}\dotsb x_n^{\alpha_{in}}x^{\beta_j} \\
&\ + a_i x_1^{\alpha_{i1}}x_2^{\alpha_{i2}}p_{\alpha_{i3}, \sigma_{i4}^{\alpha_{i4}} (\dotsb (\sigma_{in}^{\alpha_{in}}(b_j)))} x_4^{\alpha_{i4}}\dotsb x_n^{\alpha_{in}}x^{\beta_j}\\
&\ + \dotsb + a_i x_1^{\alpha_{i1}}x_2^{\alpha_{i2}} \dotsb x_{i(n-2)}^{\alpha_{i(n-2)}}p_{\alpha_{i(n-1)}, \sigma_{in}^{\alpha_{in}}(b_j)}x_n^{\alpha_{in}}x^{\beta_j} \\
&\ + a_i x_1^{\alpha_{i1}}\dotsb x_{i(n-1)}^{\alpha_{i(n-1)}}p_{\alpha_{in}, b_j}x^{\beta_j}.
\end{align*}
\end{itemize}
\end{proposition}

In the setting of SPBW extensions, Reyes and Su\'arez \cite{ReyesSuarez2018-3} defined the $\Sigma$-rigidness condition as follows: If $\Sigma=\{\sigma_1,\dots,\sigma_n\}$ is a family of endomorphisms of $R$, then $\Sigma$ is called a \emph{rigid endomorphisms family} if $a\sigma^{\alpha}(a)= 0$ implies that $a = 0$, where $a\in R$ and $\alpha\in \mathbb{N}^n$, and $R$ is a $\Sigma$-\emph{rigid} if there exists a rigid endomorphisms family $\Sigma$ of $R$ \cite[Definition 3.1]{ReyesSuarez2018-3}. Hashemi et al. \cite{HashemiKhalilAlhevaz2017} and Reyes and Su\'arez \cite {ReyesSuarez2018UMA} introduced independently $(\Sigma, \Delta)$-{\em compatible rings} as a generalization of $(\sigma, \delta)$-compatible rings. Examples, ring and module theoretic properties of these objects have been investigated \cite{HashemiKhalilAlhevaz2017, HashemiKhalilAlhevaz2019, LouzariReyes2020, LunquenJingwang2011, OuyangBirkenmeier2012, ReyesSuarez2018UMA, ReyesSuarez2020}. 

Nikmehr and Azadi \cite{NikmehrAzadi2020} considered the $\sigma$-compatibility condition for their study of the diameter and the girth of the nilpotent graph $\Gamma_N(R[x;\sigma])$. As expected, $(\Sigma,\Delta)$-compatible rings must be considered to study the interplay of ring-theoretic properties of an SPBW extension $\sigma(R)\langle x_1,\dots,x_n\rangle$ and the graph-theoretical properties of its corresponding nilpotent graph ${\rm \Gamma}_N(\sigma(R)\langle x_1,\dots,x_n\rangle)$. Before, we recall some preliminary results on compatible rings.

\begin{definition}[{\cite[Definition 3.1]{HashemiKhalilAlhevaz2017}}; {\cite[Definition 3.2]{ReyesSuarez2018UMA}}]\label{Definition3.52008}
$R$ is called {\it $\Sigma$-compatible} if for all $a, b \in R$, we obtain that $a\sigma^{\alpha}(b) = 0$ if and only if $ab = 0$, where $\alpha \in \mathbb{N}^n$; $R$ is {\it $\Delta$-compatible} if for all $a, b \in R$, we have that $ab = 0$ implies that $a\delta^{\beta}(b)=0$, where $\beta \in \mathbb{N}^n$. If $R$ is both $\Sigma$-compatible and $\Delta$-compatible, then $R$ is called {\it $(\Sigma, \Delta)$-compatible}.
\end{definition}

The following proposition is the natural generalization of \cite[Lemma 2.1]{HashemiMoussavi2005}.

\begin{proposition}[{\cite[Proposition 3.8]{ReyesSuarez2018UMA}}] \label{colosss}
If $R$ is $(\Sigma,\Delta)$-compatible, then for every $a, b
\in R$, the following assertions hold:
\begin{enumerate}
\item [\rm (1)] if $ab=0$, then $a\sigma^{\theta}(b) = \sigma^{\theta}(a)b=0$, for each $\theta\in \mathbb{N}^{n}$.

\item [\rm (2)] If $\sigma^{\beta}(a)b=0$ for some $\beta\in \mathbb{N}^{n}$, then $ab=0$.

\item [\rm (3)] If $ab=0$, then $\sigma^{\theta}(a)\delta^{\beta}(b)= \delta^{\beta}(a)\sigma^{\theta}(b) = 0$, for every $\theta, \beta\in \mathbb{N}^{n}$.
\end{enumerate}
\end{proposition}

If $R$ is $\Sigma$-rigid, it follows that $R$ is $(\Sigma, \Delta)$-compatible \cite[Proposition 3.4]{ReyesSuarez2018UMA}, but the converse does not hold \cite[Example 3.6] {ReyesSuarez2018UMA}. If $R$ is reduced, then $\Sigma$-rigid and $(\Sigma, \Delta)$-compatible coincide and if $A=\sigma(R)\langle x_1,\dots,x_n\rangle$ is a SPBW extension of $R$ then $A$ is reduced \cite[Theorem 3.9]{ReyesSuarez2018UMA}.

We present some examples of SPBW extensions over compatible rings.

\begin{example} 
 Commutative polynomial rings, PBW extensions, the algebra of linear partial differential operators, the algebra of linear partial shift operators, the algebra of linear partial difference operators, and the algebra of linear partial $q$-differential operators, and the class of diffusion algebras (for a detailed description, see \cite{Fajardoetal2020, Higuera2020, HigueraReyes2022}). Other examples include: Weyl algebras, multiplicative analogue of the Weyl algebra, some quantum Weyl algebras as $A_2(J_{a,b})$, the quantum algebra $U'(\mathfrak{so}(3, \Bbbk))$, the family of 3-dimensional skew polynomial algebras; Woronowicz algebra $W_v(\mathfrak{sl}(2, \Bbbk))$, the complex algebra $V_q({\mathfrak{sl}}_3(C))$, the Hayashi algebra $W_q(J)$, $q$-Heisenberg algebra ${\bf H}_n(q)$, and others.
\end{example}

Reyes and Su\'arez \cite{ReyesSuarez2020} introduced the {\em weak} $(\Sigma, \Delta)$-{\em compatible rings} as a generalization of $(\Sigma, \Delta)$-compatible rings and {\em weak} $(\sigma, \delta)$-{\em compatible rings} considered by Ouyang and Liu \cite{LunquenJingwang2011}. A ring $R$ is called {\it weak $\Sigma$-compatible} if for each $a, b \in R$ we obtain that $a\sigma^{\alpha}(b)\in {\rm nil}(R)$ if and only if $ab \in {\rm nil}(R)$, where $\alpha \in \mathbb{N}^n$; $R$ is {\it weak $\Delta$-compatible} if for each $a, b \in R$ we have that $ab \in {\rm nil}(R)$ implies that $a\delta^{\beta}(b)\in {\rm nil}(R)$, where $\beta \in \mathbb{N}^n$. If $R$ is both weak $\Sigma$-compatible and weak $\Delta$-compatible, then $R$ is called {\it weak $(\Sigma, \Delta)$-compatible}. 

Ouyang and Liu \cite{LunquenJingwang2011} characterized the nilpotent elements of a skew polynomial ring over a weak $(\sigma, \delta)$-compatible ring \cite[Lemma 2.13]{LunquenJingwang2011}. Reyes and Su\'arez \cite{ReyesSuarez2020} generalized this result for SPBW extensions.

\begin{proposition}[{\cite[Theorem 4.6]{ReyesSuarez2020}}]\label{ReyesSuarez2019-2Theorem4.6} If $R$ is a weak $(\Sigma,\Delta)$-compatible and NI ring, then $f = \sum_{i=0}^m a_iX_i \in {\rm nil}(\sigma(R)\langle x_1,\dots,x_n\rangle)$ if and only if $a_i\in {\rm nil}(R)$, for all $0 \le i \le m$.	
\end{proposition}

Proposition \ref{ReyesSuarez2019-2Theorem4.6} generalizes the lemma presented by Ouyang and Birkenmeier \cite{OuyangBirkenmeier2012} for skew polynomial rings \cite[Lemma 2.6]{OuyangBirkenmeier2012}, and Ouyang and Liu \cite{OuyangLiu2012} for differential polynomials rings \cite[Lemma 2.12]{OuyangLiu2012}. In addition, Proposition \ref{ReyesSuarez2019-2Theorem4.6} implies the following corollary.

\begin{corollary}\label{Corollarynilpotent} If $R$ is $(\Sigma,\Delta)$-compatible and NI, then the following statements hold:
\begin{enumerate}
\item [\rm (1)] ${\rm nil}(\sigma(R)\langle x_1,\dots,x_n\rangle)$ is an ideal and ${\rm nil}(\sigma(R)\langle x_1,\dots,x_n\rangle) = {\rm nil}(R)\langle x_1,\dotsc, x_n\rangle$.

\item [\rm (2)] $f = \sum_{i=0}^m a_iX_i \in {\rm nil}(\sigma(R)\langle x_1,\dots,x_n\rangle)$ if and only if $a_i\in {\rm nil}(R)$, for each $0 \le i \le n$. \label{MSc2.3.21}
\end{enumerate}
\end{corollary}

\section{Nilpotent graph theory of SPBW extensions}\label{Nilpotentgraph}

This section contains the original results of the paper. We characterize the diameter and girth of the nilpotent graph of SPBW extensions over $(\Sigma, \Delta)$-compatible and 2-primal rings. We start by recalling some preliminary results.

Redmond \cite{Redmond2002} defined the zero-divisor graph $\Gamma(R)$, and proved that it is always a connected graph with diameter at most 3 for any ring $R$. This fundamental result is an important tool in the study of zero-divisor graphs.

\begin{proposition}[{\cite[Theorem 3.2]{Redmond2002}}]\label{RedmondTheorem3.2} $\Gamma(R)$ is connected and ${\rm diam}(\Gamma (R))\leq 3$.
\end{proposition}

Nikmehr and Azadi \cite{NikmehrAzadi2020} showed that if $R$ is an NI ring with at least one non-zero nilpotent element, the corresponding graph $\Gamma_N(R)$ is connected, having diameter bounded by two and girth equal to either three or infinite. In addition, $\Gamma_N(R)$ is a complete graph when $R$ is isomorphic to the direct product $\mathbb{Z}_2 \times \mathbb{Z}_2$.

\begin{proposition}\label{Remark2.6Theorem2.7}
\begin{enumerate}
    \item[{\rm (1)}] \cite[Remark 2.6]{NikmehrAzadi2020} If $R$ is an NI ring and $Z_N(R)$ contains at least one non-zero nilpotent element, then $\Gamma_N(R)$ is connected, ${\rm diam}(\Gamma_N (R))\leq 2$, and $gr(\Gamma_N (R)) = 3$ or $\infty$.
    
    \item[{\rm (2)}] \cite[Theorem 2.7]{NikmehrAzadi2020} $\Gamma_N(R)$ is complete if and only if $R \cong \mathbb{Z}_2 \times \mathbb{Z}_2$.
    \end{enumerate}
\end{proposition}

As is well known, if $R$ is reduced then every minimal prime ideal of $R$ is completely prime \cite[Lemma 1]{Bell1968}, and each minimal prime ideal is a union of annihilators \cite[Lemma 1.5]{Krempa1996}. Furthermore, if $P$ is a minimal prime ideal of a reduced and $(\sigma, \delta)$-compatible ring $R$, then $P[x;\sigma,\delta]$ is a minimal prime ideal of $R[x;\sigma,\delta]$ \cite[Lemma 2.5]{WangChen2018}. The second author has studied the minimal prime ideals of SPBW extensions \cite{LouzariReyes2020, NinoReyes2020}.

Theorem \ref{PBWTheorem2.9} characterizes the diameter of the nilpotent graph of SPBW extensions over $(\Sigma, \Delta)$-compatible 2-primal rings, showing that when $R$ has exactly two minimal prime ideals, then ${\rm diam}(\Gamma_N(\sigma(R)\langle x_1,\dots,x_n\rangle)) = 2$. We denote by $|X|$ the cardinality of any set $X$. The following theorem generalizes \cite[Theorem  2.9]{NikmehrAzadi2020}.

\begin{theorem}\label{PBWTheorem2.9}
If $R$ is $(\Sigma, \Delta)$-compatible, 2-primal and has exactly two minimal prime ideals, then 
\[
{\rm diam}(\Gamma_N(\sigma(R)\langle x_1,\dots,x_n\rangle))= 2.
\]
\end{theorem}
\begin{proof}
Let $A=\sigma(R)\langle x_1,\dots,x_n\rangle$ be an SPBW extension of $R$. If $R$ is a non-reduced ring then 
$\rm{diam}(\Gamma_N(A)) \leq 2$ by Corollary \ref{Corollarynilpotent} 
and Proposition \ref{Remark2.6Theorem2.7}(1). In addition, since $A \ncong \mathbb{Z}_2 \times \mathbb{Z}_2$, we obtain that $\Gamma_N(A)$ is not complete by Proposition \ref{Remark2.6Theorem2.7}(2), whence $\rm{diam}(\Gamma_N(A)) \geq 2$ and therefore $\rm{diam}(\Gamma_N(A)) = 2$.

If $R$ is reduced then $A$ is reduced by \cite[Theorem 3.9]{ReyesSuarez2018UMA}, which implies that $Z_N(A)=Z(A)$. Additionally, if $P, P'$ are the only two minimal prime ideals of $R$, then $PA$ and $P'A$ are the only two minimal prime ideals of $A$ by \cite[Proposition 4.4]{LouzariReyes2020}, and thus $Z(A)=PA \cup P'A$. Therefore, $Z(A)$ is not an ideal of $A$, and either $|PA| \geq 3$ or $|P'A| \geq 3$.

Let $f, g \in PA$ such that $f\neq g$ and $fg = 0$. If $A$ is reduced then $A$ is reversible, and since $fg=0$ it follows that $fAg = 0$. Thus, $f \in P'A$ or $g \in P'A$, which is a contradiction because $PA \neq P'A$. In this way, $d(f,g) \geq 2$ which implies that $\rm{diam}(\Gamma(A)) \geq 2$. If $A$ is reduced then ${\rm diam}(\Gamma_N(A)) = {\rm diam}(\Gamma(A))$, and $\rm{diam}(\Gamma(A)) \leq 2$ by \cite[Theorem 3.9 (3)]{AbdiTalebi2024}, and therefore we conclude that $\rm{diam}(\Gamma_N(A)) = 2$.
\end{proof} 

\begin{corollary}[{\cite[Theorem  2.9]{NikmehrAzadi2020}}]
If $R$ is symmetric, $\sigma$-compatible and has exactly two minimal primes,
then ${\rm diam}(\Gamma_N(R[x;\sigma])) = 2$.
\end{corollary}

Under compatibility conditions, Theorem \ref{PBWTheorem2.10} characterizes the diameter of nilpotent graphs of SPBW extensions over 2-primal rings, showing that $2 \le \text{diam}(\Gamma_N(\sigma(R)\langle x_1,...,x_n\rangle) \le 3$. The proof distinguishes between non-reduced rings where diameter two occurs, and reduced rings where techniques from zero-divisor graph theory yield the upper bound three. Theorem \ref{PBWTheorem2.10} is the corresponding generalization of \cite[Theorem  2.10]{NikmehrAzadi2020}.

\begin{theorem}\label{PBWTheorem2.10} If $R$ is $(\Sigma, \Delta)$-compatible and 2-primal, then \[2 \leq {\rm diam}(\Gamma_N (\sigma(R)\langle x_1,\dots,x_n\rangle)) \leq 3.\]
\end{theorem}
\begin{proof}
If $A=\sigma(R)\langle x_1,\dots,x_n\rangle$ is an SPBW extension of $R$, then it can be seen that ${\rm diam}(\Gamma_N (R)) \leq {\rm diam}(\Gamma_N (A))$. Thus, if $R \ncong \mathbb{Z}_2 \times \mathbb{Z}_2$ then $2 \leq {\rm diam}(\Gamma_N (R))$ by Proposition \ref{Remark2.6Theorem2.7}(2), which yields that $2 \leq {\rm diam}(\Gamma_N (A))$.  

Now, if $R \cong \mathbb{Z}_2 \times \mathbb{Z}_2$, then $R$ is reduced with exactly two minimal prime ideals, whence ${\rm diam}(\Gamma_N(A)) = 2$ by Theorem \ref{PBWTheorem2.9}.  

Next, we show that ${\rm diam}(\Gamma_N(A)) \leq 3$. To prove this, we consider the following two cases:  
\begin{enumerate}[label=\Roman*.]
    \item [\rm (1)] If $R$ is non-reduced then ${\rm diam}(\Gamma_N (A)) \leq 2$ by Corollary \ref{Corollarynilpotent} and Proposition \ref{Remark2.6Theorem2.7}(1). In addition, ${\rm diam}(\Gamma_N (A)) \geq 2$ by Proposition \ref{Remark2.6Theorem2.7}(2), whence ${\rm diam}(\Gamma_N (A)) = 2 \leq 3$. 
    
    \item [\rm (2)] If $R$ is reduced, so is $A$ by \cite[Theorem 3.9]{ReyesSuarez2018UMA} and then ${\rm diam}(\Gamma_N(A))={\rm diam}(\Gamma(A))$.  Proposition \ref{RedmondTheorem3.2} guarantees that ${\rm diam}(\Gamma(A)) \leq 3$.
\end{enumerate}
\end{proof}

\begin{corollary}[{\cite[Theorem  2.10]{NikmehrAzadi2020}}] If $R$ is symmetric and $\sigma$-compatible then 
\[2 \leq {\rm diam}(\Gamma_N (R[x;\sigma])) \leq 3.\]
\end{corollary}

If $R$ is $2$-primal, then Theorem \ref{PBWTheorem2.11} establishes three fundamental properties of $\Gamma_N(R)$: connectivity, diameter bounds, and cycle length restrictions. The proof distinguishes between reduced rings (where $\Gamma_N(R) = \Gamma(R)$) separately from non-reduced cases (where nilpotent elements guarantee shorter cycles). The following theorem generalizes \cite[Theorem  2.11]{NikmehrAzadi2020}.

\begin{theorem}\label{PBWTheorem2.11}
If $R$ is $2$-primal then the following assertions hold:
\begin{itemize}
    \item[{\rm (1)}] $\Gamma_N (R)$ is connected.
    
    \item[{\rm (2)}] ${\rm diam}(\Gamma_N (R)) \leq 3$.
    
    \item[{\rm (3)}] If $\Gamma_N (R)$ contains a cycle then $gr(\Gamma_N (R)) \leq 4$.
    
    \item[{\rm (4)}] If $R$ is a non-reduced ring then $gr(\Gamma_N (R)) = 3$.
\end{itemize}
\end{theorem}
\begin{proof}
Suppose that $R$ is a reduced ring. Then $\Gamma_N(R)= \Gamma(R) $, and it yields that $\Gamma_N (R)$
is a connected graph, ${\rm diam}(\Gamma_N (R)) \leq 3$ and $gr(\Gamma_N (R)) \leq 4$ by Proposition \ref{RedmondTheorem3.2}. 

On the other hand, if $R$ is a non-reduced ring then $\Gamma_N(R)$ is connected,
${\rm diam}(\Gamma_N(R)) \leq 2$ and $gr(\Gamma_N(R)) = 3$ by Proposition \ref{Remark2.6Theorem2.7}(1) and Corollary \ref{Corollarynilpotent}. 
\end{proof}

\begin{corollary}[{\cite[Theorem  2.11]{NikmehrAzadi2020}}]
If $R$ is symmetric then the following hold:
\begin{itemize}
    \item[{\rm (1)}] $\Gamma_N (R)$ is connected.
    
    \item[{\rm (2)}] ${\rm diam}(\Gamma_N (R)) \leq 3$.
    
    \item[{\rm (3)}] If $\Gamma_N (R)$ contains a cycle then $gr(\Gamma_N (R)) \leq 4$. Moreover, if $R$ is a non-reduced ring then $gr(\Gamma_N (R)) = 3$.
\end{itemize}
\end{corollary}

Lemma \ref{PBWLemma2.12} establishes that for SPBW extensions over reduced $(\Sigma, \Delta)$-compatible rings with non-trivial nilpotent zero-divisors, the nilpotent graph satisfies $gr(\Gamma_N (\sigma(R)\langle x_1,\dots,x_n\rangle)) \le 4$. The following lemma extends \cite[Lemma 2.12]{NikmehrAzadi2020}.

\begin{lemma}\label{PBWLemma2.12}
If $R$ is $(\Sigma, \Delta)$-compatible reduced and $Z_N^{*}(R) \neq \emptyset$ then \[gr(\Gamma_N (\sigma(R)\langle x_1,\dots,x_n\rangle)) \leq 4.\]
\end{lemma} 
\begin{proof}
Assume that $A=\sigma(R)\langle x_1,\dots,x_n\rangle$ is an SPBW extension of $R$. If $R$ is reduced and $Z_N(R)^{*} \neq \emptyset$ then there exist non-zero elements $a, b \in Z_N (R)$ such that $ab = 0$, which implies that for all $\alpha, \beta \in \mathbb{N}^n$, we obtain that $b\sigma^{\alpha}(a) = a\sigma^{\beta}(b) = 0$ and $\sigma^{\theta}(a)\delta^{\beta}(b)= \delta^{\beta}(a)\sigma^{\theta}(b) = 0$ by Proposition \ref{colosss}. In this way, we may consider the cycle $a - bx^{\alpha} - ax^{\beta} - b - a$ of
length four in the nilpotent graph $\Gamma_N (A)$ whence $gr(\Gamma_N (A)) \leq 4$.
\end{proof}

Theorem \ref{PBWTheorem2.13} establishes that if $R$ is $2$-primal then $gr(\Gamma_N (R)) \geq gr(\Gamma_N (\sigma(R)\langle x_1,\dots,x_n\rangle))$, with equality when $\Gamma_N(R)$ contains cycles, and shows that ${gr}(\Gamma_N(R)) = 3$ occurs for non-reduced rings, while reduced rings satisfy ${gr}(\Gamma_N(R)) = 4$ when the $\Gamma_N(\sigma(R)\langle x_1,\dots,x_n\rangle)$ has a cycle of length four. 

The following theorem generalizes \cite[Theorem 2.13]{NikmehrAzadi2020}.

\begin{theorem}\label{PBWTheorem2.13}
If $R$ is $(\Sigma, \Delta)$-compatible and 2-primal, then 
\[
gr(\Gamma_N (R)) \geq gr(\Gamma_N (\sigma(R)\langle x_1,\dots,x_n\rangle)).
\]

In addition, if $\Gamma_N(R)$ contains a cycle, then we have that 
\[
gr(\Gamma_N (R)) = gr(\Gamma_N(\sigma(R)\langle x_1,\dots,x_n\rangle)).
\]
\end{theorem}
\begin{proof}
If $Z_N^{*}(R) = \emptyset$, then ${gr}(\Gamma_N (R)) = {gr}(\Gamma_N (A))=\infty$. Suppose that $Z_N^{*}(R) \neq \emptyset$ and let $A=\sigma(R)\langle x_1,\dots,x_n\rangle$ be an SPBW extension of $R$. It is not difficult to prove that if $\Gamma_N(R)$ is a subgraph of $\Gamma_N(A)$, then ${gr}(\Gamma_N(R)) \geq {gr}(\Gamma_N(A))$.  

If $R$ is $2$-primal but not reduced, then ${gr}(\Gamma_N(R)) = 3$ by Theorem \ref{PBWTheorem2.11}(4). Since ${gr}(\Gamma_N (R)) \geq {gr}(\Gamma_N (A))$, we have that ${gr}(\Gamma_N(A)) \le {gr}(\Gamma_N(R)) = 3$. In addition, girth is defined as the length of the shortest cycle, and a cycle of length 2 would require two distinct nilpotent elements $f, g$ such that $fg\in {\rm nil}(A)$, which would already form a cycle of length 3 in $\Gamma_N(A)$. Thus, ${gr}(\Gamma_N (A))$ cannot be less than 3, and hence ${gr}(\Gamma_N(A)) = {gr}(\Gamma_N(R)) = 3$

If $R$ is reduced, we have that ${gr}(\Gamma_N(A)) \leq 4$ by Lemma \ref{PBWLemma2.12}. Suppose that ${gr}(\Gamma_N(A)) = 4$. If the graph $\Gamma_N(R)$ contains a cycle, we obtain that ${gr}(\Gamma_N(R)) \leq 4$ by Theorem \ref{PBWTheorem2.11}(3), and since ${gr}(\Gamma_N(R)) \geq {gr}(\Gamma_N(A)) = 4$, we conclude that ${gr}(\Gamma_N(R)) = 4$.  

Assume that ${gr}(\Gamma_N(A)) = 3$ and let $f - g - h - f$ be a cycle of length three, where $f, g, h$ are non-zero elements of $A$ with leading coefficient $a_n, b_n$ and $c_n$, respectively. If $fg = 0$ then $a_n\sigma^{\alpha_n}(b_m) = 0$, and thus $a_nb_m = 0$ by $(\Sigma,\Delta)$-compatibility of $R$. Similarly, one can seen that $b_mc_l = 0$ and $c_la_n = 0$. If $R$ is reduced then $a_n$, $b_m$, and $c_l$ are distinct, and thus we may consider the cycle $a_n - b_m - c_l - a_n$, which has length three whence ${gr}(\Gamma_N(R)) = 3$.
\end{proof}

\begin{corollary}[{\cite[Theorem  2.13]{NikmehrAzadi2020}}]
If $R$ is symmetric and $\sigma$-compatible, then $gr(\Gamma_N (R)) \geq gr(\Gamma_N (R[x;\sigma]))$.
In addition, if $\Gamma_N (R)$ contains a cycle then $gr(\Gamma_N (R)) = gr(\Gamma_N(R[x;\sigma]))$.
\end{corollary}

Theorem \ref{Corollary2.14Nikmehr} characterizes the girth of the nilpotent graph $\Gamma_N(\sigma(R)\langle x_1,\dots,x_n\rangle)$, where $R$ is $(\Sigma, \Delta)$-compatible 2-primal, $gr(\Gamma_N(R)) = \infty$ and $Z_N^*(R) \neq \emptyset$. This provides a complete classification for such extensions. The following theorem generalizes \cite[Corollary  2.14]{NikmehrAzadi2020}.

\begin{theorem}\label{Corollary2.14Nikmehr} If $R$ is a $(\Sigma, \Delta)$-compatible 2-primal ring, $gr(\Gamma_N (R)) = \infty$ and $Z_N^{*} (R) \neq \emptyset$,
then we have either $|{\rm nil}(R)| = 2$ and $gr(\Gamma_N (\sigma(R)\langle x_1,\dots,x_n\rangle)) = 3$ or $R$ is a reduced ring and $gr(\Gamma_N (\sigma(R)\langle x_1,\dots,x_n\rangle)) = 4$.
\end{theorem}
\begin{proof}
Let $A=\sigma(R)\langle x_1,\dots,x_n\rangle$ be an SPBW extension of $R$. Assume that $R$ is $2$-primal but not reduced, and suppose that $|{\rm nil}(R)| \geq 3$ which implies that $|U(R)| \geq 3$. Let $a, b \in {\rm nil}(R)$ be two non-zero elements and $1 \neq c \in U(R)$. By Proposition \ref{Remark2.6Theorem2.7}(1), $a$ and $b$ are adjacent to all vertices contained $Z_N^{*} (R) = R^{*}$. In this way, $a - b - c - a$ forms a triangle and ${gr}(\Gamma_N (R)) = 3$, a contradiction. We conclude that $|{\rm nil}(R)| = 2$.  

Now, suppose $a \in {\rm nil}(R)$ with $a \neq 0$. Since $a^2 = 0$, then $a - ax^{\alpha} - ax^{\beta} - a$ forms a triangle for any non-zero element $a \in {\rm nil}(R)$ and for all $\alpha, \beta \in \mathbb{N}^n$, and hence ${gr}(\Gamma_N (A)) = 3$.  

Let us consider the other situation. If $R$ is reduced then ${gr}(\Gamma_N (A)) \leq 4$ by Theorem \ref{PBWTheorem2.13}. In addition, if ${gr}(\Gamma_N (A)) = 3$ then $f - g - h - f$ is a cycle of length three, where $f, g, h$ are non-zero elements of $A$ with leading coefficient $a_n, b_n$ and $c_n$, respectively. In Theorem \ref{PBWTheorem2.13}, we prove that the cycle $a_n - b_m - c_l - a_n$ is of length three, and then ${gr}(\Gamma_N(R)) = 3$ which is a contradiction. Thus, ${gr}(\Gamma_N (A)) = 4$.
\end{proof}

\begin{corollary}[{\cite[Corollary  2.14]{NikmehrAzadi2020}}] If $R$ is symmetric and $\sigma$-compatible, $gr(\Gamma_N (R)) = \infty$ and $Z_N^{*} (R) \neq \emptyset$,
then either $|{\rm nil}(R)| = 2$ and $gr(\Gamma_N (R[x;\sigma])) = 3$ or $R$ is reduced and $gr(\Gamma_N (R[x;\sigma])) = 4$.
\end{corollary}

\section{Examples}\label{Examplespaper}
The importance of our results is appreciated when we extend their application to algebraic structures that are more general than those considered by Nikmehr and Khojasteh \cite{Nikmehretal2013}, that is, some noncommutative rings which cannot be expressed as skew polynomial rings $R[x;\sigma]$. In this section, we consider several families of rings that have been studied in the literature which are subfamilies of skew PBW extensions. Of course, the list of examples is not exhaustive.

\begin{definition}[{\cite{BellSmith1990}; \cite[Definition C4.3]{Rosenberg1995}}]\label{monito}
A $3$-{\em dimensional algebra} $\mathcal{A}$ is a $\Bbbk$-algebra generated by the indeterminates $x, y, z$ subject to the relations $yz-\alpha zy=\lambda$, $zx-\beta xz=\mu$, and $xy-\gamma yx=\nu$, where $\lambda,\mu,\nu \in \Bbbk x+\Bbbk y+\Bbbk z+\Bbbk$, and $\alpha, \beta, \gamma \in \Bbbk^{*}$. If the set $\{x^iy^jz^k\mid i,j,k\geq 0\}$ forms a $\Bbbk$-basis of $\mathcal{A}$, then it is called a \textit{3-dimensional skew polynomial $\Bbbk$-algebra}
\end{definition}

Up to isomorphism, there exist fifteen 3-dimensional skew polynomial $\Bbbk$-algebras \cite[Theorem C4.3.1]{Rosenberg1995}. Different authors have studied ring-theoretical and geometrical properties of these algebras (e.g., \cite{Redman1999}).  
It follows from Definition \ref{monito} that every $3$-dimensional skew polynomial algebra is an SPBW extension over the field $\Bbbk$.  By Theorem \ref{PBWTheorem2.10}, the diameter of the nilpotent graph of $\mathcal{A}$ must be either 2 or 3. Additionally, the girth of the nilpotent graph of $\mathcal{A}$ cannot exceed that of the nilpotent graph of $\Bbbk$ by Theorem \ref{PBWTheorem2.13}.

\begin{example}
   \cite[p. 30]{Fajardoetal2020} The {\em diffusion algebra} $A$ is generated by $2n$ indeterminates  $D_i, x_i$ over $\Bbbk$ with $1 \leq i \leq n$ and subjects to the relations
\begin{center}
    $\displaystyle x_ix_j = x_jx_i,\ \ x_iD_j = D_jx_i, \ \ 1 \leq i, j \leq n$, \\
$c_{ij}D_iD_j - c_{ji}D_jD_i = x_jD_i - x_iD_j , \ \  i < j, \ \ c_{ij} , c_{ji} \in \Bbbk^{*}$.
\end{center}
According to Definition \ref{gpbwextension}, the algebra $A$ can be seen as an SPBW extension of $\Bbbk[x_1, \dots , x_n]$, but not as a PBW extension or an iterated skew polynomial ring of injective type. Theorem \ref{PBWTheorem2.10} provides bounds on the diameter of the nilpotent graph $\Gamma_N(A)$, proving that the maximal distance between any two non-zero nilpotent elements in $A$ is either $2$ or $3$, due to $\Bbbk[x_1, \dots, x_n]$ being $(\Sigma, \Delta)$-compatible and 2-primal. Moreover, Theorem \ref{PBWTheorem2.13} establishes that the girth of $\Gamma_N(A)$ does not exceed that of $\Gamma_N(\Bbbk[x_1, \dots, x_n])$. 
\end{example}

\begin{example}
Algebras whose generators satisfy quadratic relations such as Clifford algebras, Weyl-Heisenberg algebras, and Sklyanin algebras, play an important role in analysis and mathematical physics. Motivated by these facts, Golovashkin and Maximov \cite{GolovashkinMaximov2005} considered the algebras $Q(a, b, c)$, with two generators $x$ and $y$, defined by the quadratic relations
\begin{equation}\label{GolovashkinMaximov2005(1)}
    yx = ax^2 + bxy + cy^2,
\end{equation}
    
where the coefficients $a$, $b$, and $c$ belong to an arbitrary field $\Bbbk$ of characteristic zero. They presented conditions on these elements under which such an algebra has a PBW basis of the form $\mathcal{B}=\{x^m y^n \mid m, n \in \mathbb{N}\}$. Golovashkin and Maximov given a necessary and sufficient condition for a PBW basis as above to exist when $ac + b \neq 0$ \cite[Section 3]{GolovashkinMaximov2005}. If $ac + b = 0$, they proved that if $b \neq 0, -1$ then $Q(a, b, c)$ has a PBW basis, and if $b = -1$ then the set $\mathcal{B}$ is linearly independent but do not form a PBW basis of $Q(a, b, c)$ \cite[Section 5]{GolovashkinMaximov2005}.
    
We can prove that if $a$, $b$, and $c$ are not all zero, then $Q(a, b, c)$ is neither a skew polynomial ring of $\Bbbk$, $\Bbbk[x]$, nor of $\Bbbk[y]$. If $b \neq 0$ and $c = 0$, one can verify that $Q(a, b, c)$ is an SPBW extension of $\Bbbk[x]$, that is $Q(a, b, c) \cong \sigma(\Bbbk[x])\langle y\rangle$. Theorem \ref{PBWTheorem2.10} determines that the diameter of the nilpotent graph of $Q(a,b,0)$ must be either 2 or 3, providing bounds on how nilpotent elements relate to each other in these quadratic algebras. Moreover, the girth of the nilpotent graph of $Q(a,b,0)$ cannot exceed that of the nilpotent graph of $\Bbbk[x]$ by Theorem \ref{PBWTheorem2.13}. 

\end{example}
    
\begin{example}
Jordan \cite{Jordan2000} introduced a class of iterated skew polynomial rings $R(B, \sigma, c, p)$, known as \emph{ambiskew polynomial rings}, which include different examples of noncommutative algebras. These structures have been investigated by Jordan at various levels of generality in several papers \cite{Jordan1993, Jordan1993b, Jordan1995, Jordan1995b}. We recall the treatment of ambiskew polynomial rings in the framework appropriate to down-up algebras with $\beta \neq 0$.

Let $B$ be a commutative $\Bbbk$-algebra, $\sigma$ be a $\Bbbk$-automorphism of $B$, and elements $c \in B$ and $p \in \Bbbk^*$. If $S$ is the skew polynomial ring $B[x; \sigma^{-1}]$, and extend $\sigma$ to $S$ by setting $\sigma(x) = p x$, then there is a $\sigma$-derivation $\delta$ of $S$ such that $\delta(B) = 0$ and $\delta(x) = c$ \cite[p.~41]{Cohn1985}. The \emph{ambiskew polynomial ring} $R = R(B, \sigma, c, p)$ is the ring $S[y; \sigma, \delta]$, which satisfies the following relations:
\begin{equation}
    yx - p xy = c, \quad \text{and, for all } b \in B, \quad x b = \sigma^{-1}(b) x \quad \text{and} \quad y b = \sigma(b) y.
\end{equation}

Equivalently, $R$ can be presented as $R = B[y; \sigma][x; \sigma^{-1}, \delta']$, where $\sigma(y) = p^{-1} y$, $\delta'(B) = 0$, and $\delta'(y) = -p^{-1} c$, yielding the relation $xy - p^{-1} yx = -p^{-1} c$. If we interpret the relation $x b = \sigma^{-1}(b) x$ as $b x = x \sigma(b)$, we observe that the definition involves twists from both sides using $\sigma$; this motivates the name of these objects. It is well-known that every generalized Weyl algebra is isomorphic to a factor of an ambiskew polynomial ring, and the ring $R(B, \sigma, c, p)$ is isomorphic to the generalized Weyl algebra $B[w](\sigma, w)$, where $\sigma$ is extended to $B[w]$ by setting $\sigma(w) = p w + \sigma(c)$. It is not difficult to show that ambiskew polynomial rings are skew PBW extensions over $B$, that is $R(B, \sigma, c, p) \cong \sigma(B)\langle y, x \rangle$. If $B$ has exactly two minimal primes, then Theorem \ref{PBWTheorem2.9} establishes that ${\rm diam}(\Gamma_N(R(B, \sigma, c, p))= 2$. Theorem \ref{PBWTheorem2.10} provides bounds on the diameter of $\Gamma_N (R(B, \sigma, c, p))$, demonstrating that for a $(\Sigma, \Delta)$-compatible and 2-primal ring $B$, the maximal distance between any two nilpotent elements in $R(B, \sigma, c, p)$ is either $2$ or $3$. Furthermore, Theorem \ref{PBWTheorem2.13} proves that the girth of $\Gamma_N (R(B, \sigma, c, p))$ cannot exceed that of $\Gamma_N (B)$. Additionally, when $B$ has infinite girth but $Z_N^{*} (B) \neq \emptyset$, the girth of $\Gamma_N (R(B, \sigma, c, p))$ is three if $|{\rm nil}(B)| = 2$, and four if $B$ is reduced by Theorem \ref{Corollary2.14Nikmehr}.
\end{example}

\begin{example}
{\em Skew bi-quadratic algebras} were introduced by Bavula \cite{Bavula2023} as a framework for classifying bi-quadratic algebras on three generators having a PBW basis.

If $L_n(R)$ denotes the set lower triangular matrices with elements of $R$, $\mathcal{Z}(R)$ is the center of $R$, $\sigma = (\sigma_1, \dotsc, \sigma_n)$ is an $n$-tuple of commuting endomorphisms of $R$, $\delta = (\delta_1, \dotsc, \delta_n)$ is an $n$-tuple of $\sigma$-derivations of $R$ (i.e., $\delta_i$ is a $\sigma_i$-derivation of $R$ for all $ i \le n$), $Q = (q_{ij}) \in L_n(Z(R))$, $\mathbb{A}:= (a_{ij, k}) \in L_n(R)$ and $\mathbb{B}:= (b_{ij}) \in L_n(R)$ where $1 \leq j < i \leq n$ and $k \le n$, then {\em skew bi-quadratic algebra} ({\em SBQA}), which is denoted by $R[x_1,\dotsc, x_n;\sigma, \delta, Q, \mathbb{A}, \mathbb{B}]$ is the ring generated by $R$ and the elements $x_1, \dotsc, x_n$ subject to the relations:
\begin{align}
    x_i r &= \sigma_i(r)x_i + \delta_i(r),\quad \text{for all}\ 1 \le i \le n, \text{ and } r \in R, \label{Bavula2023(1)} \\
    x_i x_j - q_{ij} x_j x_i &= \sum_{k=1}^{n} a_{ij, k} x_k + b_{ij},\quad \text{for all } j < i.\label{Bavula2023(2)}
\end{align}

If $\sigma_i$ is the identity homomorphism of $R$ and $\delta_i$ is the zero derivation of $R$, for all $1 \le i \le n$, then $R[x_1, \dotsc, x_n; Q, \mathbb{A}, \mathbb{B}]$ is called a {\em bi-quadratic algebra} ({\em BQA}). A skew bi-quadratic algebra has a {\em PBW basis} if $A = \bigoplus_{\alpha \in \mathbb{N}^{n}} R x^{\alpha}$ where $x^{\alpha} = x_1^{\alpha_1} \dotsb x_n^{\alpha_n}$.
\end{example}

The definition of SBQA makes it clear that $R[x_1, \dotsc, x_n;\sigma, \delta, Q, \mathbb{A}, \mathbb{B}] \cong \sigma(R)\langle x_1,\dotsc, x_n\rangle$. Moreover, SPBW extensions are more general than skew bi-quadratic algebras, since they do not require the $\sigma$'s to commute (cf. \cite[Definition 2.1]{AcostaLezamaReyes2015}) nor impose any membership conditions on $q_{ij}$ and $b_{ij}$. In this way, if $R$ has exactly two minimal prime ideals, then Theorem \ref{PBWTheorem2.9} shows that ${\rm diam}(\Gamma_N(R[x_1, \dotsc, x_n;\sigma, \delta, Q, \mathbb{A}, \mathbb{B}])= 2$. Theorem \ref{PBWTheorem2.10} establishes bounds for the diameter of $\Gamma_N (R[x_1, \dotsc, x_n;\sigma, \delta, Q, \mathbb{A}, \mathbb{B}])$, showing that if $R$ is $(\Sigma, \Delta)$-compatible and 2-primal, then the maximum distance between any two nilpotent elements of $R[x_1, \dotsc, x_n;\sigma, \delta, Q, \mathbb{A}, \mathbb{B}]$ is either $2$ or $3$. Theorem \ref{PBWTheorem2.13} shows that $\Gamma_N (R[x_1, \dotsc, x_n;\sigma, \delta, Q, \mathbb{A}, \mathbb{B}])$ cannot have larger girth than $\Gamma_N (R)$. If $R$ has infinite girth but $Z_N^{*} (R) \neq \emptyset$, then $\Gamma_N (R[x_1, \dotsc, x_n;\sigma, \delta, Q, \mathbb{A}, \mathbb{B}])$ have girth three when $|{\rm nil}(R)| = 2$, or four when $R$ is reduced by Theorem \ref{Corollary2.14Nikmehr}.

\section{Future work}

Motivated by the concept of weak compatibility introduced by Reyes and Su\'arez \cite{ReyesSuarez2020}, it is natural to turn our attention to SPBW extensions defined over weak $(\Sigma,\Delta)$-compatible NI rings, with the aim of characterizing both the diameter and girth of their nilpotent graphs.

In this direction, Hashemi et al. \cite{HashemiAmirjanAlhevaz2017} presented a characterization of the possible diameters of $\Gamma(R[x;\sigma,\delta])$ in terms of the diameter of $\Gamma(R)$. Furthermore, they demonstrated how certain properties of $R$ can be determined when the diameters of both $R$ and $R[x;\sigma, \delta]$ are known. A subsequent research objective would be to investigate similar diameter characterizations for nilpotent graphs and to explore how the properties of $R$ can be deduced from knowledge of both a weak $(\Sigma,\Delta)$-compatible ring $R$ and an SPBW extension of $R$.

\section{Declarations}\label{Declarations}

The first author was supported by Direcci\'on Ciencias B\'asicas, Universidad ECCI - sede Bogot\'a, while the second author was supported by Facultad de Ciencias, Universidad Nacional de Colombia - Sede Bogot\'a.

All authors declare that they have no conflicts of interest.



\end{document}